\documentclass[12pt, reqno]{amsart}
\usepackage{color}
\usepackage{hyperref}
\usepackage{graphicx}
\usepackage{amsmath}% amstex ?
\usepackage{amsthm}
\usepackage{amssymb,bbm}%
\usepackage[numbers, square]{natbib}
\usepackage{mathrsfs}
\usepackage{enumerate}
\usepackage{dsfont}
\usepackage{bbm}

  \evensidemargin
\oddsidemargin
\numberwithin{equation}{section}

\newcommand{\Sc}{\mathcal{S}}
\newcommand{\Lc}{\mathcal{L}}

\newcommand{\E}{\mathbb E}
\newcommand{\Pc}{\mathcal P}
\newcommand{\vol}{\operatorname{Vol}}
\newcommand{\Ac}{\mathcal{A}}

\newcommand{\Crit}{\mathcal{C}}
\newcommand{\prob}{\mathcal{P}r}

\newcommand{\R}{\mathbb{R}}

\newcommand{\C}{\mathbb{C}}
\newcommand{\Z}{\mathbb{Z}}

\newcommand{\Hc}{\mathcal{H}}

\newcommand{\M}{\mathcal{M}}

\theoremstyle{plain}
\newtheorem{theorem}{Theorem}[section]
\newtheorem{proposition}[theorem]{Proposition}
\newtheorem{lemma}[theorem]{Lemma}

\theoremstyle{definition}
\newtheorem{definition}[theorem]{Definition}

\newcommand{\nod}{\mathcal{N}}

\newcommand{\Zc}{\mathcal{Z}}
\newcommand{\Dc}{\mathcal{D}}

\begin{document}

\title[Betti numbers]{On the expected Betti numbers of the nodal set of random fields}

\author{Igor Wigman}
\email{igor.wigman@kcl.ac.uk}
\address{Department of Mathematics, King's College London}

\date{\today}

\begin{abstract}
This note concerns the asymptotics of the expected total Betti numbers of the nodal set
for an important class of Gaussian ensembles of random fields on Riemannian manifolds.
By working with the limit random field defined on the Euclidean space we were able to obtain
a locally precise asymptotic result, though due to the possible positive
contribution of large {\em percolating} components this does not allow to infer a global result.
As a by-product of our analysis, we refine the lower bound of Gayet-Welschinger for the
important Kostlan ensemble of random polynomials and its generalisation to K\"{a}hler manifolds.
\end{abstract}
	
\maketitle

\section{Introduction}

\subsection{Betti numbers for random fields: Euclidean case}

Let $F:\R^{d}\rightarrow\R$ be a centred stationary Gaussian random field, $d\ge 2$. The {\em nodal set of $F$} is its (random) zero set $$\Zc_{F}:=F^{-1}(0)=\{x\in\R^{d}:\: F(x)=0\}\subseteq\R^{d};$$ assuming $F$ is sufficiently {\em smooth} and {\em non-degenerate} (or regular),
its connected components (``nodal components of $F$") are a.s. either closed $(d-1)$-manifolds or smooth infinite hypersurfaces (``percolating components").
One way to study the topology of $\Zc_{F}$, a central research thread in the recent few years,
is by restricting $F$ to a large centred
ball $B(R)=\{x\in\R^{d}:\: \|x\|<R \}$, and then investigate the restricted nodal set $\widetilde{\Zc_{F}}(R):=F^{-1}(0)\cap B(R)$ as $R\rightarrow\infty$.
The set $\widetilde{\Zc_{F}}(R)$ consists of the union of the a.s. smooth closed nodal components of $\Zc_{F}$
lying entirely in $B(R)$,
and the fractions of nodal components of $F$ intersecting $\partial B(R)$; note that, by intersecting
with $B(R)$, the components intersecting $\partial B(R)$, finite or percolating,
might break into $2$ or more connected components, or fail to be closed.

It follows as a by-product of the precise analysis due to Nazarov-Sodin ~\cite{sodin_lec_notes,nazarov_sodin} that, under very mild assumptions on $F$ to be discussed below, mainly concerning its smoothness and non-degeneracy, with high probability {\em most}
of the components of $\Zc_{F}$ fall into the former, rather than the latter,
category (see \eqref{eq:nod numb conv mean} below).
That is, for $R$ large, with high probability, most of the components of $\Zc_{F}$ intersecting $B(R)$ are lying entirely within $B(R)$.
Setting $$\Zc_{F}(R):=\bigcup\limits_{\gamma\subseteq B(R)}\gamma$$
to be the union of all the nodal components $\gamma$ of $F$ lying entirely in $B(R)$, the first primary concern
of this note is in the topology of $\Zc_{F}(R)$, and, in particular, the Betti numbers of
$\Zc_{F}(R)$ as $R\rightarrow\infty$, more precisely, the asymptotics of their expected values.

For $0\le i \le d-1$ the corresponding Betti number $b_{i}(\cdot )$ is the dimension of the $i$'th
homology group, so that a.s.
\begin{equation}
\label{eq:betai def}
\beta_{i}(R)=\beta_{F;i}(R):=b_{i}(\Zc_{F}(R)) = \sum\limits_{\gamma\subseteq \Zc_{F}(R)} b_{i}(\gamma),
\end{equation}
summation over all nodal components $\gamma$ lying in $\Zc_{F}(R)$. For example, $\beta_{0}=:\nod_{F}(R)$ is the
total number of connected components $\gamma\subseteq \Zc_{F}(R)$ (``nodal count") analysed by Nazarov-Sodin,
and $$\beta_{i}(R) = \beta_{d-1-i}(R)$$ by
Poincar\'{e} duality. To be able to state Nazarov-Sodin's results we need to introduce the following axioms;
by convention they are expressed in terms of the spectral measure rather than $F$ or its covariance function.

\begin{definition}[Axioms $(\rho 1)-(\rho 4)$ on $F$]
\label{def:axioms rho1-4}

Let $F:\R^{d}\rightarrow\R$ be a Gaussian stationary random field, $$r_{F}(x-y)=r_{F}(x,y):=\E[F(x)\cdot F(y)]$$
the covariance function of $F$, and $\rho=\rho_{F}$ be its spectral measure, i.e. the Fourier transform of $r_{F}$
on $\R^{d}$.

\begin{enumerate}

\item $F$ satisfies $(\rho 1)$ if the measure $\rho$ has no atoms.

\item $F$ satisfies $(\rho 2)$ if for some $p>6$, $$\int\limits_{\R^{d}}\|\lambda\|^{p}d\rho(\lambda)<\infty.$$

\item $F$ satisfies $(\rho 3)$ if the support of $\rho$ does not lie in a linear hyperplane of $\R^{d}$.

\item $F$ satisfies $(\rho 4)$ if the interior of the support of $\rho$ is non-empty.

\end{enumerate}

\end{definition}

Axioms $(\rho 1)$, $(\rho 2)$ and $(\rho 3)$ ensure that the action of translations on $\R^{d}$ is ergodic, a.s. sufficient smoothness
of $F$, and non-degeneracy of $F$ understood in proper sense, respectively. Axiom $(\rho 4)$ implies that any smooth function
belongs to the support of the law of $F$, which, in turn, will yield the positivity of the number of nodal components, and positive representation
of every topological type of nodal components. 

\vspace{2mm}

Recall that $\nod_{F}(R)=\beta_{0}(R)$ is the number of nodal components of $F$ entirely lying in $B(R)$, and let
$V_{d}$ be the volume of the unit $d$-ball, and $\vol B(R)=V_{d}\cdot R^{d}$ be the volume of the radius $R$ ball in $\R^{d}$. Nazarov and Sodin
~\cite{sodin_lec_notes,nazarov_sodin} proved
that if $F$ satisfies $(\rho 1)-(\rho 3)$, then there exists a constant $c_{NS}=c_{NS}(\rho_{F})$ (``Nazarov-Sodin constant") so that
$\frac{\nod_{F}(R)}{\vol B(R)}$ converges to $c_{NS}$, both in mean and a.s. That is, as $R\rightarrow\infty$,
\begin{equation}
\label{eq:nod numb conv mean}
\E\left[\left|\frac{\nod_{F}(R)}{\vol B(R)} - c_{NS} \right|\right] \rightarrow 0,
\end{equation}
so that, in particular,
\begin{equation}
\label{eq:exp nod numb asymp}
\E[\nod_{F}(R)] = c_{NS}\cdot \vol B(R)+o(R^{d}).
\end{equation}
They also showed that imposing $(\rho 4)$ is sufficient (but not necessary)
for the strict positivity of $c_{NS}$, and found other very mild sufficient
conditions on $\rho$, so that $c_{NS}>0$. The validity of the asymptotic \eqref{eq:exp nod numb asymp}
for the expected nodal count was extended ~\cite{kurlberg2018variation} to hold without imposing the ergodicity axiom $(\rho 1)$,
with $c_{NS}=c_{NS}(\rho_{F})$ appropriately generalised,
also establishing a stronger estimate for the error term as compared to the r.h.s. of \eqref{eq:exp nod numb asymp}.

One might think that endowing the ``larger" components with the same weight $1$ as the ``smaller" components might be ``discriminatory"
towards the larger ones, so that separating the counts based on the components' topology ~\cite{sarnak_wigman16} or geometry
~\cite{beliaev2018volume} would provide an adequate response for the alleged discrimination. These nevertheless do not address the
important question of the {\em total} Betti number $\beta_{i}$, the main difficulty being that the individual Betti number
$b_{i}(\gamma)$ of a nodal component $\gamma$ of $F$ is not bounded,
even under the assumption that $\gamma \subseteq B(R)$ is entirely lying inside a compact domain. Despite this, we will be
able to resolve this difficulty by controlling from above the total Betti number via Morse Theory ~\cite{milnor1963morse},
an approach already
pursued by Gayet-Welschinger ~\cite{gayet2016betti} (see \S\ref{sec:proofs outline} below for a more detailed explanation).

\begin{theorem}
\label{thm:Betti asymp Euclid}
Let $F:\R^{d}\rightarrow\R$ be a centred Gaussian random field, satisfying axioms $(\rho 2)$ and $(\rho 3)$ of 
Definition \ref{def:axioms rho1-4}, $d\ge 2$, and $0\le i \le d-1$. Then

\begin{enumerate}[a.]

\item
There exists a number $c_{i}=c_{F;i}\ge 0$ so that
\begin{equation}
\label{eq:exp Betti R^d}
\E[\beta_{i}(R)] = c_{i}\cdot\vol B(R) + o_{R\rightarrow\infty}(R^{d}).
\end{equation}

\item If, in addition, $F$ satisfies $(\rho 1)$, then convergence \eqref{eq:exp Betti R^d} could be extended to hold in mean, i.e.
\begin{equation}
\label{eq:conv Betti L1 R^d}
\E\left[\left| \frac{\beta_{i}(R)}{\vol B(R)} - c_{i}  \right|\right] \rightarrow 0
\end{equation}
as $R\rightarrow\infty$.

\item
\label{it:rho4=>cNS>0}
Further, if $F$ satisfies the axiom $(\rho 4)$ (in addition to $(\rho 2)$ and $(\rho 3)$, but not $(\rho 1)$),
then $c_{i}>0$. The same conclusion holds for the important Berry's monochromatic isotropic random waves in arbitrary dimensions
(``Berry's random wave model").

\end{enumerate}

\end{theorem}

\subsection{Motivation and background}

The Betti numbers of both the nodal and the excursion sets of Gaussian random fields
serve as their important topological descriptor, and are therefore addressed in both mathematics and experimental
physics literature, in particular cosmology ~\cite{park2013betti}.
From the complex geometry perspective Gayet and Welschinger ~\cite{gayet2016betti} studied the distribution of the total Betti numbers of the zero set for the Kostlan
Gaussian ensemble of degree $n$ random homogeneous polynomials
on the $d$-dimensional projective space, and their generalisation to K\"{a}hler manifolds,
$n\rightarrow\infty$. In the projective coordinates $x=[x_{0}:\ldots : x_{d}]\in \R \Pc^{d}$ we may write
\begin{equation}
\label{eq:Pn Kostlan def}
P_{n}(x) = \sum\limits_{|j|=n}\sqrt{{n \choose j}} a_{j} x^{j},
\end{equation}
where $j=(j_{0},\ldots,j_{d})$, $|j|=\sum\limits_{i=0}^{d}j_{i}$, $x^{j}=x_{0}^{j_{0}}\cdot \ldots \cdot x_{d}^{j_{d}}$,
$ {n \choose j} = \frac{n!}{j_{0}!\cdot \ldots\cdot j_{d}!}$, and $\{a_{j}\}$ are standard Gaussian i.i.d.
By the homogeneity of $P_{n}$, its zero set makes sense on the projective space. The Kostlan (also referred to as ``Shub-Smale")
ensemble is an important model of random polynomials, uniquely invariant w.r.t. unitary transformations on $\C \Pc^{d}$.
Restricted to the unit sphere
$\Sc^{d}\subseteq \R^{d+1}$, the random fields $P_{n}$ are defined by the covariance function
\begin{equation*}
\E[P_{n}(x)\cdot P_{n}(y)] = \langle x,y\rangle^{n} = \left(\cos(\theta(x,y))\right)^{n},
\end{equation*}
where $x,y\in \Sc^{d}$, the inner product $\langle \cdot,\cdot\rangle$ is inherited from $\R^{d+1}$, and
$\theta(\cdot,\cdot)$ is the angle between two points on $\R^{d+1}$.

Upon scaling by $\sqrt{n}$ (the meaning is explained in Definition \ref{def:loc lim} below),
the Kostlan polynomials \eqref{eq:Pn Kostlan def} admit
~\cite[\S 2.5.4]{sodin_lec_notes}, locally uniformly, a (stationary isotropic) limit random field on $\R^{d}$, namely
the Bargmann-Fock ensemble defined by the ``Gaussian" covariance kernel
\begin{equation}
\label{eq:kappa Gauss BF}
\kappa(x):=e^{-x^{2}/2},
\end{equation}
see also ~\cite{beffara2017percolation,beliaev2017russo}. This indicates that
one should expect the Betti numbers to be of order of magnitude $\approx n^{d/2}$.
That this is so is supported by Gayet-Welschinger's upper bounds ~\cite{gayet2016betti}
\begin{equation*}
\E[b_{i}(P_{n}^{-1}(0))] \le A_{i} n^{d/2}
\end{equation*}
with some semi-explicit $A_{i}>0$, and the subsequent lower bounds ~\cite{gayet2014lower}
\begin{equation}
\label{eq:GW lower bound}
\E[b_{i}(P_{n}^{-1}(0))] \ge a_{i} n^{d/2},
\end{equation}
$a_{i}>0$, but to our best knowledge the important question of the true asymptotic behaviour of $b_{i}(P_{n}^{-1}(0))$
is still open.

\subsection{Betti numbers for Gaussian ensembles on Riemannian manifolds}

Since
$\kappa$ of \eqref{eq:kappa Gauss BF} (or, rather, its Fourier transform) easily satisfies all Nazarov-Sodin's axioms $(\rho 1)-(\rho 4)$
of Definition \ref{def:axioms rho1-4}, one wishes to invoke Theorem \ref{thm:Betti asymp Euclid}
with the Bargmann-Fock field in place of $F$, and try to deduce the results analogous to \eqref{eq:exp Betti R^d}
for the Betti numbers of the nodal set of $P_{n}$ in \eqref{eq:Pn Kostlan def}. This is precisely the purpose
of Theorem \ref{thm:tot Betti numb loc} below, valid in a scenario of {\em local translation invariant limits}, far
more general than merely the Kostlan ensemble, whose introduction is our next goal.

\vspace{2mm}

Let $\M$ be a compact Riemannian $d$-manifold, and $\{f_{L}\}_{L\in\Lc}$ be a family of {\em smooth} Gaussian random fields $f_{L}:\M\rightarrow\R$,
where the index $L$ attains a {\em discrete} set $\Lc$, and $K_{L}(\cdot,\cdot)$ the covariance function corresponding to $f_{L}$,
so that $$K_{L}(x,y)=\E[f_{L}(x)\cdot f_{L}(y)];$$ the parameter $L$ should be thought of as the scaling factor, generalising
the rolse of $\sqrt{n}$ for the Kostlan ensemble. We scale $f_{L}$ restricted to a sufficiently small neighbourhood of
a point $x\in \M$, so that the exponential map $\exp_{x}(\cdot):T_{x}\M\rightarrow\M$ is well defined.
We define
\begin{equation}
\label{eq:fx,L scal def}
f_{x,L}(u):= f_{L}(\exp_{x}(u/L),
\end{equation}
with covariance
$$K_{x,L}(u,v) := K_{L}(\exp_{x}(u/L),\exp_{x}(v/L))$$ with $|u|,|v|<L\cdot r$ with $r$ sufficiently small, uniformly
with $x\in\M$, allowing $u,v$ to grow with $L\rightarrow\infty$.

\begin{definition}[Local translation invariant limits, cf. ~{\cite[Definition 2 on p. 6]{nazarov_sodin}}]
\label{def:loc lim}

We say that the Gaussian ensemble $\{f_{L}\}_{L\in\Lc}$ possesses local translation invariant limits, if for almost all
$x\in \M$ there exists a positive definite function $K_{x}:\R^{d}\rightarrow\R$, so that for all $R>0$,
\begin{equation}
\label{eq:covar scal lim}
\lim\limits_{L\rightarrow\infty}\sup\limits_{|u|,|v|\le R}\left| K_{x,L}(u,v)-K_{x}(u-v)\right| \rightarrow 0.
\end{equation}
\end{definition}

Important examples of Gaussian ensembles possessing
translation invariant local limits include (but not limited to) Kostlan's ensemble
\eqref{eq:Pn Kostlan def}
of random homogeneous polynomials, and Gaussian band-limited functions ~\cite{sarnak_wigman16}, i.e. Gaussian superpositions of Laplace
eigenfunctions corresponding to eigenvalues lying in an energy window. For manifolds with spectral degeneracy, such as
the sphere and the torus (and $d$-cube with boundary), the {\em monochromatic} random waves (i.e. Gaussian superpositions of Laplace
eigenfunctions belonging to the same eigenspace) are a particular case of band-limited functions; two of the most interesting cases
are those of random spherical harmonics (random Laplace eigenfunctions on the round unit $d$-sphere)
~\cite{wigman2009distribution,wigman2010fluctuations}, and ``Arithmetic Random Waves"
(random Laplace eigenfunctions on the standard $d$-torus) ~\cite{oravecz2008leray,krishnapur2013nodal}.

In all the said examples of Gaussian ensembles on manifolds of our particular interest the scaling limit $K_{x}$ (and
the associate Gaussian random field on $\R^{d}$) was independent of $x$, and the limit in \eqref{eq:covar scal lim} is uniform, attained
in a strong quantitative form, see the discussion in ~\cite[\S 2.1]{beliaev2019mean}. We will also need the following,
more technical concepts of uniform smoothness and non-degeneracy for $\{f_{L}\}$, introduced in
~{\cite[definitions 2-3, p. 14-15]{sodin_lec_notes}}.

\begin{definition}[Smoothness and non-degeneracy]\ \\

\begin{enumerate}

\item We say that $\{f_{L}\}$ is $C^{3-}$ smooth if
for every $0<R<\infty$,
\begin{equation*}
\limsup\limits_{L\rightarrow\infty}\sup\left\{ |\partial_{u}^{i}\partial^{j}_{v} K_{x,L}(u,v)|:\: |i|,|j|\le 3;\;x\in \M, \|u\|,\|v\|\le R  \right\}< \infty.
\end{equation*}

\item We say that $\{f_{L}\}$ is non-degenerate if for every $0<R<\infty$
\begin{equation*}
\liminf\limits_{L\rightarrow\infty}\inf\left\{ \E\left[\partial_{\xi}f_{x,L}(u)^{2}\right]:\:  \xi\in\Sc^{d-1},\, x\in \M,\,\|u\|\le R  \right\}>0.
\end{equation*}

\end{enumerate}

\end{definition}

Let $\{f_{L}\}_{L\in\Lc}$ be a $C^{3-}$ smooth, non-degenerate,
Gaussian ensemble possessing translation invariant local limits $K_{x}$, corresponding
to Gaussian random fields on $R^{d}$ with spectral measure $\rho_{x}$, satisfying axioms
$(\rho 1)-(\rho 3)$. Denote $\nod(f_{L};x,R/L)$ to be the
number of nodal components of $f_{L}$ lying entirely in the geodesic ball $B_{x}(R/L)\subseteq\M$,
and $\nod(f_{L})$ to be the {\em total} number of the nodal components of $f_{L}$ on $\M$.
In this settings Nazarov-Sodin ~\cite{sodin_lec_notes,nazarov_sodin} proved that
\begin{equation}
\label{eq:nod comp loc Riemann}
\lim\limits_{R\rightarrow\infty}\limsup\limits_{L\rightarrow\infty}
\E\left[ \left| \frac{\nod(f_{L};x,R/L)}{\vol B(R)} -  c_{NS}(\rho_{x})    \right| \right] = 0,
\end{equation}
with $c_{NS}(\cdot)$ same as in \eqref{eq:nod numb conv mean}.

For the total number $\nod(f_{L})$ they glued the local results \eqref{eq:nod comp loc Riemann}, to deduce,
on invoking a two-parameter analogue of Egorov's Theorem yielding the {\em almost uniform} convergence of \eqref{eq:nod comp loc Riemann}
w.r.t. $x$, that
\begin{equation}
\label{eq:nod comp glob Riemann}
\lim\limits_{R\rightarrow\infty}
\E\left[ \left| \frac{\nod(f_{L})}{V_{d}L^{d}} -  \nu    \right| \right] \rightarrow 0,
\end{equation}
holds with $$\nu:=\int\limits_{\M}c_{NS}(\rho_{x})dx.$$
In particular, \eqref{eq:nod comp glob Riemann} yields
\begin{equation}
\label{eq:exp nod comp glob}
\E[\nod(f_{L})] = V_{d}\nu \cdot L^{d} + o(L^{d}),
\end{equation}

\vspace{2mm}

As it was mentioned above, in practice, in many applications, the scaling limit $K_{x}(\cdot)\equiv K(\cdot)$ does not depend on
$x$, so that, assuming w.l.o.g. that $\vol(\M)=1$, the asymptotic constant $\nu$ in \eqref{eq:nod comp glob Riemann}
(and \eqref{eq:exp nod comp glob}) is $\nu=c_{NS}(\rho)$, where $\rho$ is the Fourier transform of $K$.
In this situation, in accordance with Theorem \ref{thm:Betti asymp Euclid}\ref{it:rho4=>cNS>0}, $\nu=c_{NS}>0$ is positive, if
$(\rho 4)$ is satisfied.
The following result extends \eqref{eq:nod comp loc Riemann} to arbitrary Betti numbers.

\begin{theorem}
\label{thm:tot Betti numb loc}

Let $\{f_{L}\}_{L\in\Lc}$ be a $C^{3-}$ smooth, non-degenerate,
Gaussian ensemble, $x\in \M$ satisfying \eqref{eq:covar scal lim} with some $K_{x}$
satisfying axioms $(\rho 1)-(\rho 3)$, and $0 \le i \le d-1$.
Denote $\beta_{i;L}(x,R/L)=\beta_{i}(f_{L};x,R/L)$ to be the total $i$'th Betti
number of the union of all components of $f_{L}^{-1}(0)$ entirely contained
in the geodesic ball $B_{x}(R/L)$. Then for every $\epsilon>0$
\begin{equation}
\label{eq:betti loc prob conv ci}
\lim\limits_{R\rightarrow\infty}\limsup\limits_{L\rightarrow\infty}
\prob\left\{ \left| \frac{\beta_{i;L}(x,R/L)}{\vol B(R)}- c_{i} \right| > \epsilon \right\} = 0.
\end{equation}
where $c_{i}$ is the same as in \eqref{eq:conv Betti L1 R^d}, corresponding to the random field defined
by $K_{x}$.
\end{theorem}

Theorem \ref{thm:tot Betti numb loc} asserts that the random variables $\left\{\frac{\beta_{i;L}(x,R/L)}{\vol B(R)}\right\}_{L\in\Lc}$
converge in probability to $c_{i}$, in the double limit $L\rightarrow\infty$, and then $R\rightarrow\infty$.
One would be tempted to try to deduce the convergence in mean for
the same setting, the main obstacle being that $\beta_{i;L}(x,R/L)$ is not bounded, and, in principle, a small probability
event might contribute positively to the expectation of $\beta_{i;L}(x,R/L)$. While it is plausible (if not likely) that
a handy bound on the variance (or the second moment), such as ~\cite{estrade2016number,muirhead2019second},
for the critical points number would rule this out and establish the desired $L^{1}$-convergence in this, or, perhaps, slightly more restrictive scenario, we will not pursue this direction in the present manuscript, for the sake of
keeping it compact.

\vspace{2mm}

Theorem \ref{thm:tot Betti numb loc} applied on the Kostlan ensemble \eqref{eq:Pn Kostlan def} of
random polynomials, in particular, recovers Gayet-Welschinger's later lower bound \eqref{eq:GW lower bound},
but, finer, with high probability, it prescribes the asymptotics of the total Betti numbers of
all the components lying in geodesic balls of radius slightly above $1/\sqrt{n}$, and hence, in this case, one might
think of Theorem \ref{thm:tot Betti numb loc} as a refinement of \eqref{eq:GW lower bound}.
It would be desirable to determine the true asymptotic law of
$\E[b_{i}(P_{n}^{-1}(0))]$ (hopefully, for the more general scenario),
though the possibility of giant (``percolating") components is a genuine consideration,
and, if our present understanding of this subtlety is correct ~\cite{beliaev2019mean},
then, to resolve the asymptotics of $\E[b_{i}(P_{n}^{-1}(0))]$
the question whether they consume a positive proportion of the total Betti numbers cannot be possibly avoided.
In fact, it is likely that for $d\ge 3$, with high probability, there exists a single percolating
component consuming a high proportion of the space, and contributing positively to the Betti numbers, as
found numerically by Barnett-Jin (presented within ~\cite{sarnak_wigman16}), and explained by P. Sarnak
~\cite{Sa}, with the use of percolating vs. non-percolating random fields (see ~\cite[\S 1.2]{beliaev2019mean} for more details,
and also the discussion in \S\ref{sec:proofs outline} below).

\subsection{Acknowledgement}

It is a pleasure to thank P. Sarnak and M. Sodin for their comments on the presented proofs
of the main results, D. Panov for freely sharing his expertise on various aspects of the presented material,
Z. Rudnick for his support and encouragement regarding this work and his valuable comments on an earlier version
of this manuscript, and Z. Kabluchko for pointing out
the superadditive ergodic theorem ~\cite[Theorem 2.14, page 210]{krengel_book}. The author of this manuscript
is grateful to D. Beliaev and S. Muirhead for many stimulating conversations concerning subjects of high relelvance to
the presented research.
The research leading to these results has received funding from the European Research Council under
the European Union's Seventh Framework Programme (FP7/2007-2013), ERC grant agreement n$^{\text{o}}$ 335141.

\vspace{-1mm}

\section{Outline of the proofs and discussion}
\label{sec:proofs outline}

\vspace{-1mm}

\subsection{Outline of the proofs of the principle results}

The principal novel result of this manuscript is Theorem \ref{thm:Betti asymp Euclid}.
Theorem \ref{thm:Betti asymp Euclid} {\em given}, the proof of Theorem \ref{thm:tot Betti numb loc} does not differ significantly from the proof of ~\cite[Theorem 5]{sodin_lec_notes} given ~\cite[Theorem 1]{sodin_lec_notes}. The key observation here is that while passing from the Euclidean random field $F_{x}$ to its perturbed Riemannian version $f_{x,L}$ in the vicinity of $x\in\M$, the topology of its nodal set
is preserved on a high probability {\em stable} event, to be constructed, and hence so is its $i$'th Betti number.
In fact, this was the conclusion from the argument presented in ~\cite[Theorem 6.2]{sarnak_wigman16}
that will reconstructed in \S\ref{sec:proof loc Riem}, alas briefly, for the sake of completeness.

\begin{figure}[ht]
\centering
\includegraphics[height=70mm]{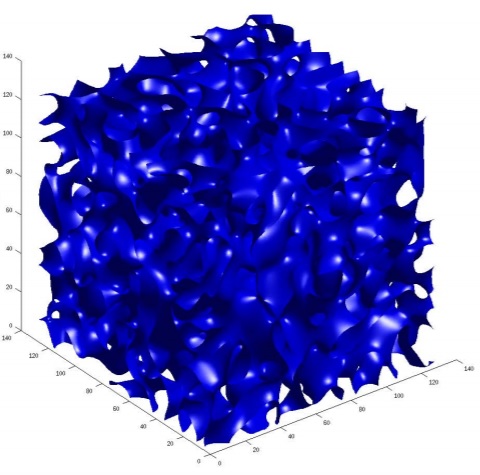}
\includegraphics[height=70mm]{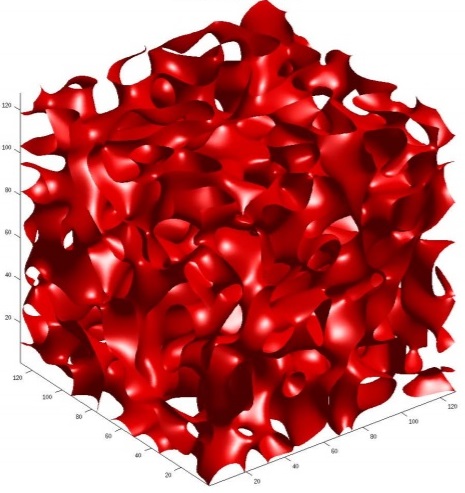}
\caption{Computer simulations by A. Barnett.
Left: Giant percolating nodal components for $3$-dimensional monochromatic isotropic waves.
Right: Analogous picture for the ``Real Fubini-Study" (a random ensemble of homogeneous polynomials, with
different law as compared to Kostlan's ensemble).}
\label{fig:Barnett 3d giant}
\end{figure}

To address the asymptotic expected nodal count $\nod_{F}(R) = \beta_{F;0}(R)$, Nazarov-Sodin have developed the so-called {\em Integral Geometric sandwich}. The idea is that one bounds, $\nod_{F}(R)$ from below using $\nod_{\cdot}(r)$, of radii $0<r<R$ much smaller than $R$ (``fixed"),
and $F$ translated (equivalently, shifter radius-$r$ ball), and from above using a version of $\nod_{F}(r)$, where, rather than counting
nodal components lying entirely in $B(r)$ (or its shift), we also include those components intersecting its boundary $\partial B(r)$.
By invoking ergodic methods one shows that both these bounds converge to the same limit, and in its turn this yields automatically both the asymptotics for the expected nodal count, and the convergence in mean.

Unfortunately, since we endow each nodal component $\gamma$ with the, possibly unbounded,
weight $b_{i}(\gamma)$, the upper bound in the sandwich does not seemingly yield a useful result. We bypass this major obstacle by using
a global bound on the expected Betti numbers via Morse Theory (and the Kac-Rice method), and then establishing an asymptotics
for the expected number. Rather than working with arbitrary chosen ``fixed" radii $r>0$, we only work with ``good" radii, defined
so that these numbers are ``almost maximising" the expected Betti numbers, so that we can infer the same for all the sufficiently
big radii $R>r$ (see \eqref{eq:eta def limsup} and \eqref{eq:betai/r^d>eta-eps}). In hindsight, we interpret working with the good radii
as ``miraculously" eliminating the possible fluctuations in the contribution to the Betti numbers of the giant percolating domains.
Once the asymptotics for the expected Betti number has been determined, we tour de force working with the good radii
to also yield the convergence in mean, with the help of the ergodic assumption $(\rho 1)$.

\vspace{2mm}

Another possible strategy for proving results like Theorem \ref{thm:Betti asymp Euclid} is by
observing that, by naturally extending the definition of $\beta_{i}$ to smooth domains $\Dc\subseteq\R^{d}$ as
\begin{equation*}
\beta_{i}(\Dc)=\beta_{i;F}(\Dc):= \sum\limits_{\gamma\subseteq \Dc} b_{i}(\gamma),
\end{equation*}
with summation over the (random) nodal components of $F$ lying in $\Dc$,
$\beta_{i}(\cdot)$ is made into a {\em super-additive random variable}, i.e. for all $\Dc_{1},\ldots,\Dc_{k}\subseteq \R^{d}$ pairwise disjoint
domains, the inequality
\begin{equation*}
\beta_{i}\left(\bigcup\limits_{j=1}^{k}\Dc_{j}\right)\ge \sum\limits_{j=1}^{k} \beta_{i}(\Dc_{j})
\end{equation*}
holds. It then might be tempting to apply the superadditive ergodic theorem ~\cite[Theorem 2.14, page 210]{krengel_book} (and its finer
version ~\cite[p.~165]{nguyen}) on $\beta_{i}$. However, in this manuscript we will present a direct and explicit
treatise of this subject.

\subsection{Discussion}

As it was mentioned above, a straightforward application of \ref{thm:tot Betti numb loc} on the Kostlan's ensemble of random homogenous
polynomials, in particular implies the lower bound \eqref{eq:GW lower bound} for the total expected Betti number for
this ensemble due to Gayet-Welschinger,
and its generalisations for K\"{a}hler manifolds. Our argument is entirely different as compared to Gayet-Welschinger's:
rather than working with the finite degree polynomials ~\eqref{eq:Pn Kostlan def}, as in ~\cite{gayet2014lower}, we first prove the result
for the limit Bargmann-Fock random field on $\R^{d}$ (Theorem \ref{thm:Betti asymp Euclid}), and then deduce the result by a perturbative procedure
following Nazarov-Sodin (Theorem \ref{thm:tot Betti numb loc}).

\vspace{2mm}

It is crucial to determine whether the global asymptotics
\begin{equation*}
\E[\beta_{i;L}] \sim c_{i}\vol(\M)\cdot L^{d},
\end{equation*}
expected from its local probabilistic version \eqref{eq:betti loc prob conv ci}, could be extended
to hold for the total expected Betti number of $f^{-1}$ in some scenario, inclusive of all the motivational examples.
Such a result would indicate that no giant ``percolating" components, not lying inside any {\em macroscopic}
(or slightly bigger) geodesic balls exist, contributing positively to the Betti numbers.
In fact some numerics due to Barnett-Jin (presented within ~\cite{sarnak_wigman16})
support the contrary for $d\ge 3$, as argued by Sarnak ~\cite{Sa}, see
Figure \ref{fig:Barnett 3d giant}, and also ~\cite[\S 2.1]{beliaev2019mean}.
To our best knowledge, at this stage this question is entirely open, save for the results on $\beta_{0;L}$
(and $\beta_{d-1;L}$) due to Nazarov-Sodin.

\section{Proof of Theorem \ref{thm:Betti asymp Euclid}}

\subsection{Auxiliary lemmas}

Recall that $\beta_{i}(R)=b_{i}(\Zc_{F}(R))$ is defined in \eqref{eq:betai def}, and for $x\in \R^{d}$, $R>0$,
introduce
\begin{equation}
\label{eq:bi loc sum def}
\beta_{i}(x;R)=\beta_{F;i}(x,R) := \sum\limits_{\gamma\subseteq \Z_{F}\cap B_{x}(R)} b_{i}(\gamma),
\end{equation}
summation over all nodal components of $F$ contained in the shifted ball $B_{x}(R)$, or, equivalently
\begin{equation*}
\beta_{F;i}(x,R) = \beta_{T_{x}F;i}(R),
\end{equation*}
where $T_{x}$ acts by translation $(T_{x}F)(\cdot)=F(\cdot-x)$.

\begin{lemma}[Integral-Geometric sandwich, lower bound; cf. ~{\cite[Lemma 1]{sodin_lec_notes}}]
\label{lem:Int Geom sand}
For every $0<r<R$ we have the following inequality
\begin{equation}
\label{eq:Int Geom sand}
\frac{1}{\vol B(r)}\int\limits_{B(R-r)}\beta_{i}(x;r)dx   \le \beta_{i}(R).
\end{equation}

\end{lemma}

\begin{proof}
Since if a nodal component of $F$ is contained in $B_{x}(r)$ for some $x\in B(R-r)$, then $\gamma \subseteq B(R)$, we may
invert the order of summation and integration to write:
\begin{equation*}
\begin{split}
\frac{1}{\vol B(r)}\int\limits_{B(R-r)}\beta_{i}(x;r)dx &=
\frac{1}{\vol B(r)}\int\limits_{B(R-r)}\sum\limits_{\gamma\subseteq \Zc(F)}\mathbbm{1}_{\gamma\subseteq B_{x}(r)}\cdot b_{i}(\gamma)dx
\\&=\frac{1}{\vol B(r)}\sum\limits_{\gamma\subseteq \Zc(F)\cap B(R)} b_{i}(\gamma)\cdot \vol\{x\in B(R-r):\: \gamma\subseteq B_{x}(r) \}
\\&\le \sum\limits_{\gamma\subseteq \Zc(F)\cap B(R)} b_{i}(\gamma)=b_{i}(R),
\end{split}
\end{equation*}
since
\begin{equation*}
\{x\in B(R-r):\: \gamma\subseteq B_{x}(r) \} = \bigcap\limits_{y\in \gamma}B_{y}(r)
\end{equation*}
is of volume $\le \vol B(r)$.
\end{proof}

The intuition behind the inequality \eqref{eq:Int Geom sand} is, in essence, the convexity of the involved quantities.
One can also establish the upper bound counterpart of \eqref{eq:Int Geom sand}, whence will need to introduce the
$\beta^{*}_{\cdot}(\cdot;\cdot)$
analogue, where the summation range on the r.h.s. \eqref{eq:bi loc sum def} is extended to nodal components $\gamma$
merely {\em intersecting} $B_{x}(R)$. However, since the contribution of a single nodal component to the total
Betti number is not bounded, and is expected to be {\em huge} for percolating components, we did not find a useful
way to exploit such an upper bound inequality. Instead we are going to seek for a global bound, via Kac-Rice
estimating of a relevant local quantity.

\begin{lemma}[Upper bound]
\label{lem:upper bnd loc}
Let $F$ and $i$ be as in Theorem \ref{thm:Betti asymp Euclid}. Then
\begin{equation}
\label{eq:upper bnd loc}
\limsup\limits_{R\rightarrow\infty} \frac{\E[\beta_{i}(R)]}{R^{d}} < \infty.
\end{equation}
\end{lemma}

\begin{proof}

We use Morse Theory to reduce bounding the expected Betti number $\E[\beta_{i}(R)]$ from above to a {\em local} computation,
performed with the aid of Kac-Rice method, an approach already exploited by Gayet-Welschinger ~\cite{gayet2016betti}.
Let $\gamma\subseteq \R^{d}$ be a compact closed hypersurface, and $g:\R^{d}\rightarrow\R$ a smooth function so that its
restriction $g|_{\gamma}$ to $\gamma$ is a Morse function (i.e. $g|_{\gamma}$ has no degenerate critical points).
Then, as a particular consequence of the Morse inequalities ~\cite[Theorem 5.2 (2) on p. 29]{milnor1963morse}, we have
\begin{equation*}
b_{i}(\gamma) \le \Crit_{i}(g|_{\gamma}),
\end{equation*}
where $\Crit_{i}(g|_{\gamma})$ is the number of critical points of $g|_{\gamma}$ of Morse index $i$.
Under the notation of Theorem \ref{thm:Betti asymp Euclid} it follows that
\begin{equation}
\label{eq:betai<=tot Crit}
\E[\beta_{i}(R)] \le \E[\Crit_{i}(g|_{F^{-1}(0) \cap B(R)})]\le \E[\Crit(g|_{F^{-1}(0) \cap B(R)})],
\end{equation}
the r.h.s. of \eqref{eq:betai<=tot Crit} being the total number of critical points of $g$ restricted to the nodal set of $F$
lying in $B(R)$, a local quantity that could be evaluated with the Kac-Rice method.

\vspace{2mm}

Now we evaluate the r.h.s. of \eqref{eq:betai<=tot Crit}, where we have the freedom to choose the function $g$,
so long as it is a.s. Morse restricted to $F^{-1}(0)$. As a concrete simple case, we nominate the
function $$\R^{d}\ni x=(x_{1},\ldots, x_{d})\mapsto g(x)=\|x\|^{2}=\sum\limits_{j=1}^{d}x_{i}^{2},$$
or, more generally, the family of functions $g_{p}=\|x-p\|^{2}$, $p\in \R^{d}$,
having the burden of proving that for some $p\in\R^{d}$, the restriction $g_{p}|_{F^{-1}(0)}$ of
$g_{p}$ to $F^{-1}(0)$ is Morse a.s. For this particular choice of the family $g_{p}$,
a point $x\in F^{-1}(0)\setminus \{0\}$ is a critical point of $g_{p}$, if and only if $\nabla F(x)$ is collinear to
$x-p$. Normalising $v_{1}:=\frac{x-p}{\|x-p\|}$, this is equivalent to
$\nabla F(x)\perp v_{j}$, $j=2,\ldots d$, where $\{v_{j}\}_{2\le j\le d}$ is any orthonormal basis of $v_{1}^{\perp}$,
and it is possible to make a locally smooth choice for $\{v_{j}\}_{2\le j\le d}$ as a function of $x$ (or, rather
$v_{1}$), since $\Sc^{d-1}$ admits orthogonal frames on a finite partition of the sphere into coordinate patches.

Now, by ~\cite[Lemma 6.3, Lemma 6.5]{milnor1963morse}, a critical point $x\in F^{-1}(0)$, of $g_{p}$ is degenerate,
if and only if $p=x+K^{-1}\cdot v_{1}$, with $K$ one of the (at most $d-1$) principal curvatures of $F^{-1}(0)$ at $x$ in direction $v_{1}$,
and, by Sard's Theorem \cite[Theorem 6.6]{milnor1963morse}, given a sample function $F_{\omega}$, where $\omega\in\Omega$ is a sample point in the underlying sample space $\Omega$, the collection $A_{\omega}\subseteq \R^{d}$ of all ``bad" $p$, so that $g_{p}|_{F^{-1}(0)}$ contains a 
degenerate critical point is of vanishing Lebesgue measure, i.e. 
\begin{equation}
\label{eq:mu(Aomega)=0}
\mu(A_{\omega})=0, 
\end{equation}
a.s. We are aiming at showing that there exists $p\in\R^{d}$
so that a.s. $p\notin A_{\omega}$; in fact, by the above, we will be able to conclude, via Fubini, that $\mu$-almost all $p$ will do
(and then, since, by stationarity of $F$, there is no preference of points in $\R^{d}$, we will be able to carry
out the computations with the simplest possible choice $p=0$, though the computations are not significantly more involved 
with arbitrary $p$).
To this end we introduce the set $$\Ac=\{(p,\omega):\: p\in A_{\omega}\}\subseteq \R^{d}\times \Omega$$ 
on the measurable space $\R^{d}\times \Omega$, equipped with the measure $d\lambda=d\mu(p) d\prob(\omega)$. Since there is no 
measurability issue here, an inversion of the integral
\begin{equation*}
\lambda(\Ac) = \int\limits_{\Ac}d\mu(p) d\prob(\omega) = 0, 
\end{equation*}
by \eqref{eq:mu(Aomega)=0}, yields that for $\mu$-almost all $p\in\R^{d}$, 
\begin{equation}
\label{eq:prob(p bad)=0}
\prob\{p\in A_{\omega}\} = 0.
\end{equation}

\vspace{2mm}

The above \eqref{eq:prob(p bad)=0} yields a point $p\in\R^{d}$, so that $g_{p}|_{F^{-1}(0)}$
is a.s. Morse, and, in particular \eqref{eq:betai<=tot Crit} holds a.s. with $g=g_{p}$; by the stationarity of
$F$, we may assume that $p=0$, and we take $g=g_{0}$.
Next we plan to employ the Kac-Rice method for evaluating the expected number of critical points of $g|_{F^{-1}(0)} $ as on the r.h.s. of 
\eqref{eq:betai<=tot Crit}.
Recall from above that, for this particular choice of $g$,
a point $x\in F^{-1}(0)\setminus \{0\}$ is a critical point of $g$, if and only if $\nabla F(x)$ is collinear to
$v_{1}=v_{1}(x):=\frac{x}{\|x\|}$, or, equivalently,
$\nabla F(x)\perp v_{j}$, $j=2,\ldots d$, where $\{v_{j}\}_{2\le j\le d}$ is any orthonormal basis of $v_{1}^{\perp}$.

Let
\begin{equation}
\label{eq:G(x) def}
G(x)=\left(F(x),\langle v_{2},\nabla F(x)\rangle,\ldots, \langle v_{d},\nabla F(x)\rangle\right)
\end{equation}
be the Gaussian random vector, and $C_{G}(x)$ its $d\times d$ covariance matrix.
That the joint Gaussian distribution of $G(x)$ is non-degenerate, is guaranteed by the axiom $(\rho 3)$, since this axiom yields
~\cite[\S 1.2.1]{sodin_lec_notes} the non-degeneracy of the distribution of $\nabla F(x)$ (and hence of any linear transformation
of $\nabla F(x)$ of full rank),
and $F(x)$ is statistically independent of $\nabla F(x)$.
By the Kac-Rice formula ~\cite[Theorem 6.3]{azais_wschebor}, using the non-degeneracy of the distribution of
$G(x)$ as an input, we conclude that for every $\epsilon>0$
\begin{equation}
\label{eq:exp Crit KR excise}
\E[\Crit(g|_{F^{-1}(0) \cap (B(R)\setminus B(\epsilon))})] = \int\limits_{B(R)\setminus B(\epsilon)} K_{1}(x)dx,
\end{equation}
where for $x\ne 0$, the density is defined as the Gaussian integral
\begin{equation}
\label{eq:K1 density x}
K_{1}(x) = K_{1;F}(x) = \frac{1}{(2\pi)^{d/2}\sqrt{|\det C_{G}(x)|}}\cdot \E[ |\det H_{G}(x) | \big| G(x)=0 ],
\end{equation}
and $H_{G}(\cdot)$ is the Hessian of $G$.
Next we apply the Monotone Convergence theorem on \eqref{eq:exp Crit KR excise} as $\epsilon\rightarrow\infty$, upon bearing
in mind that $x=0$ is not a zero of $F$ a.s., we obtain
\begin{equation}
\label{eq:exp Crit KR}
\E[\Crit(g|_{F^{-1}(0) \cap B(R)})] = \int\limits_{B(R)} K_{1}(x)dx,
\end{equation}
extending the definition of $K_{1}$ at $x=0$ arbitrarily.

\vspace{2mm}

In what follows we are going to show that $K_{1}(\cdot)$ is {\em bounded} on $\R^{d}$,
which, in light of \eqref{eq:exp Crit KR} is sufficient to yield \eqref{eq:upper bnd loc}, via \eqref{eq:betai<=tot Crit}.
To this end we observe that, since $F$ is stationary, the value of $K_{1}$ is defined intrinsically as a function of
$v_{1}\in \Sc^{d-1}$, no matter how $v_{j}$, $j\ge 2$ were determined, as long as they constitute an o.n.b. of $v_{1}^{\perp}$,
i.e. $$K_{1}(x)=K_{1}(x/\|x\|) = K_{1}(v_{1}),$$ despite the fact that the law of $G(x)$ does, in general, depend on the choice of the vectors
$\{v_{j}\}$ , $j\ge 2$.

The upshot is that, since,
given $v_{1}\in \Sc^{d-1}$, one can choose $\{v_{j}\}_{2\le j\le d}$ locally continuously, also determining the law of $G(x)$
in a locally continuous and non-degenerate way as a function of $v_{1}$, meaning that $ |\det C_{G}(\cdot)|>0$.
Hence $K_{1}(\cdot)$ in \eqref{eq:K1 density x} is a continuous function of $v_{1}\in \Sc^{d-1}$, and therefore it
is bounded by a constant depending only on the law of $F$ (though not necessarily defined continuously {\em at} the origin).
As it was readily mentioned,
the boundedness of $K_{1}$ is sufficient to yield the statement \eqref{eq:upper bnd loc} of Lemma \ref{lem:upper bnd loc}.

\end{proof}

The following lemma is a restatement of ~\cite[Proposition 5.2]{sarnak_wigman} for random fields
satisfying $(\rho 4)$, and of ~\cite[Theorem 1.3(i)]{canzani2019topology} for Berry's monochromatic isotropic waves
in higher dimensions, and thereupon its proof will be conveniently omitted here.

\begin{lemma}
\label{lem:lower bound}

Let $F:\R^{d}\rightarrow\R$ be a Gaussian random field,
$\Hc(d-1)$ the collection of all diffeomorphism classes of closed $(d-1)$-manifolds that have an embedding in
$\R^{d}$, and for $H\in\Hc(d-1)$ denote $\nod_{F,H}(R)$ the number of nodal components of $F$, entirely contained
in $B(R)$ and diffeomorphic to $H$. Then if $F$ either satisfies $(\rho 4)$ or it is Berry's monochromatic isotropic waves,
one has:
\begin{equation*}
\liminf\limits_{R\rightarrow\infty}\frac{\E[\nod_{F,H}(R)]}{R^{d}} > 0.
\end{equation*}

\end{lemma}

\subsection{Proof of Theorem \ref{thm:Betti asymp Euclid}}

\begin{proof}

First we aim at proving \eqref{eq:exp Betti R^d}, that will allow us to deduce \eqref{eq:conv Betti L1 R^d}, with the help of
\eqref{eq:Int Geom sand}.
Take
\begin{equation}
\label{eq:eta def limsup}
\eta:= \limsup\limits_{R\rightarrow\infty}\frac{\E[\beta_{i}(R)]}{R^{d}}.
\end{equation}
Then, necessarily $\eta<\infty$ is finite, thanks to Lemma \ref{lem:upper bnd loc}. We claim that, in fact, \eqref{eq:eta def limsup},
is a limit, whence it is sufficient to show that
\begin{equation}
\label{eq:liminf = eta}
\liminf\limits_{R\rightarrow\infty}\frac{\E[\beta_{i}(R)]}{R^{d}} \ge \eta.
\end{equation}
To this end we take $\epsilon >0$ to be an arbitrary positive number, and, by the definition of $\eta$ as a $\limsup$,
we may choose $r=r(\epsilon)>0$ so that
\begin{equation}
\label{eq:betai/r^d>eta-eps}
\frac{\E[\beta_{i}(r)]}{r^{d}} > \eta - \epsilon.
\end{equation}

We now take $R>r$, and appeal to the Integral Geometric sandwich \eqref{eq:Int Geom sand}, so that taking an expectation of both sides
of \eqref{eq:Int Geom sand} yields
\begin{equation}
\label{eq:E[bi] convexity}
\E[\beta_{i}(R)] \ge \frac{1}{\vol B(r)}\int\limits_{B(R-r)}\E[\beta_{i}(x;r)]dx = \frac{(R-r)^{d}}{r^{d}}\cdot \E[\beta_{i}(r)],
\end{equation}
by the stationarity of $F$. Substituting \eqref{eq:betai/r^d>eta-eps} into \eqref{eq:E[bi] convexity}, it follows that
\begin{equation*}
\E[\beta_{i}(R)] \ge (R-r)^{d}\cdot (\eta-\epsilon),
\end{equation*}
and hence, dividing by $R^{d}$, and taking $\liminf\limits_{R\rightarrow\infty}$ (note that $r$ is kept fixed), we obtain
\begin{equation*}
\liminf\limits_{R\rightarrow\infty}\frac{\E[\beta_{i}(R)]}{R^{d}} \ge \eta-\epsilon.
\end{equation*}
Since $\epsilon>0$ is arbitrary, this certainly implies \eqref{eq:liminf = eta}, which, as it was mentioned above,
implies that $\eta$ in \eqref{eq:eta def limsup} is a limit, a restatement of \eqref{eq:exp Betti R^d} (with
$c_{i} = \frac{\eta}{V_{d}}$).

\vspace{2mm}

Next, having proved \eqref{eq:exp Betti R^d}, we are going to deduce the convergence in mean \eqref{eq:conv Betti L1 R^d},
this time, assuming the axiom $(\rho 1)$, yielding that the action of the translations $\{T_{x}\}_{x\in\R^{d}}$ is ergodic,
proved independently by Fomin ~\cite{fomin}, Grenander ~\cite{grenander1950stochastic}, and Maruyama
~\cite{maruyama1949harmonic} (see also ~\cite[Theorem 3]{sodin_lec_notes}).
Let $0<r<R$, and denote the random variable
\begin{equation}
\label{eq:Psi i def}
\Psi_{i}(R,r)=\Psi_{i}(F;R,r) :=\frac{1}{\vol B(r)}\int\limits_{B(R-r)}\beta_{i}(x;r)dx,
\end{equation}
so that the Integral Geometric sandwich \eqref{eq:Int Geom sand} reads
\begin{equation}
\label{eq:Psi<=beta}
\Psi_{i}(R,r) \le \beta_{i}(R),
\end{equation}
and the aforementioned ergodic theorem asserts that, for $r$ fixed, as $R\rightarrow\infty$,
\begin{equation*}
\frac{1}{\vol B(R-r)}\Psi_{i}(R,r)  \rightarrow \frac{\E[\beta_{i}(r)]}{\vol B(r)}
\end{equation*}
in mean (and a.s.), so that we may deduce the same for
\begin{equation}
\label{eq:Psi/B(R)->betai/B(r)}
\frac{1}{\vol B(R)}\Psi_{i}(R,r)  \rightarrow \frac{\E[\beta_{i}(r)]}{\vol B(r)},
\end{equation}
in mean.

Now let $\epsilon>0$ be arbitrary, and use \eqref{eq:exp Betti R^d}, now at our disposal,
to choose $r=r(\epsilon)$ sufficiently large (but fixed) so that
\begin{equation}
\label{eq:beta/B(r)-ci<eps/3}
\left|\frac{\E[\beta_{i}(r)]}{\vol B(r)} -  c_{i}\right| < \frac{\epsilon}{3},
\end{equation}
and also,
\begin{equation}
\label{eq:exp(betai) Cauchy}
\left|\frac{\E[\beta_{i}(R)]}{\vol B(R)}-\frac{\E[\beta_{i}(r)]}{\vol B(r)}\right| < \frac{\epsilon}{4},
\end{equation}
for the function $$r\mapsto \frac{\E[\beta_{i}(r)]}{\vol B(r)}$$ being Cauchy as $r\rightarrow\infty$.
Next, use \eqref{eq:Psi/B(R)->betai/B(r)} in order for the inequality
\begin{equation}
\label{eq:exp(Psii-betai)<eps/3}
\E\left[\left|\frac{1}{\vol B(R)}\Psi_{i}(R,r)  - \frac{\E[\beta_{i}(r)]}{\vol B(r)}\right|\right] < \frac{\epsilon}{3},
\end{equation}
to hold, provided that $R$ is sufficiently large (depending on $r$ and $\epsilon$).
Note that, thanks to \eqref{eq:Psi<=beta}, we have
\begin{equation}
\label{eq:Ebetai-EPsii}
\begin{split}
&0\le \E\left[\left| \frac{\beta_{i}(R)}{\vol B(R)}  - \frac{1}{\vol B(R)}\Psi_{i}(R,r)\right| \right]=
\E\left[ \frac{\beta_{i}(R)}{\vol B(R)}  - \frac{1}{\vol B(R)}\Psi_{i}(R,r) \right]
\\&=\frac{\E[\beta_{i}(R)]}{\vol B(R)}  - \frac{1}{\vol B(R)}\E[\Psi_{i}(R,r)] =
\frac{\E[\beta_{i}(R)]}{\vol B(R)}  - \frac{\vol B(R-r)}{\vol B(r) \vol B(R)}\E[\beta_{i}(r)]
\\&= \frac{\E[\beta_{i}(R)]}{\vol B(R)}  - (1+o_{R\rightarrow\infty}(1)) \cdot\frac{\E[\beta_{i}(r)]}{\vol B(r)} < \frac{\epsilon}{3}
\end{split}
\end{equation}
for $R$ sufficiently large, by \eqref{eq:Psi i def}, the stationarity of $F$, and \eqref{eq:exp(betai) Cauchy}.
We consolidate all the above inequalities by using the triangle inequality to write
\begin{equation*}
\begin{split}
&\E\left[ \left|\frac{\beta_{i}(R)}{\vol B(R)}  -c_{i}\right| \right] \le
\E\left[ \frac{\beta_{i}(R)}{\vol B(R)}  - \frac{1}{\vol B(R)} \Psi_{i}(R,r) \right]
\\&+ \E\left[ \left|\frac{1}{\vol B(R)}\Psi_{i}(R,r) - \frac{\E[\beta_{i}(r)]}{\vol B(r)}\right| \right] +
\E\left[ \left|\frac{\E[\beta_{i}(r)]}{\vol B(r)} - c_{i}\right| \right] < \epsilon,
\end{split}
\end{equation*}
by \eqref{eq:beta/B(r)-ci<eps/3}, \eqref{eq:exp(Psii-betai)<eps/3} and \eqref{eq:Ebetai-EPsii}. Since $\epsilon>0$ was
an arbitrary positive number, the mean convergence \eqref{eq:conv Betti L1 R^d} is now established.
Finally, we observe that Theorem \ref{thm:Betti asymp Euclid}\ref{it:rho4=>cNS>0} is a direct consequence of Lemma \ref{lem:lower bound}.
Theorem \ref{thm:Betti asymp Euclid} is now proved.

\end{proof}

\section{Proof of Theorem \ref{thm:tot Betti numb loc}}
\label{sec:proof loc Riem}

Let $x\in\M$ be a point as postulated in Theorem \ref{thm:tot Betti numb loc},
$K_{x}$ the corresponding covariance kernel, and $F_{x}$ the centred Gaussian random field defined by $F_{x}$.
Recall that $f_{x,L}(\cdot)$, defined in \eqref{eq:fx,L scal def} on $\R^{d}$ via the identification $T_{x}(\M)\cong \R^{d}$,
is the scaled version of $f_{L}$,
converging in the limit $L\rightarrow\infty$, to $F_{x}$, with accordance to \eqref{eq:covar scal lim}.
By the manifold structure of $\M$, the exponential map $\exp_{x}:T_{x}\rightarrow\M$ is a diffeomorphism
on a sufficiently small ball $B(r)\subseteq T_{x}$, with $r>0$ independent of $x$.
Hence, for every $R>0$,
the diffeomorphism types in $B(R)\subseteq\R^{d}\cong T_{x}(\M)$ are preserved under the {\em scaled} exponential map
$$\exp_{x;L}:u\mapsto \exp_{x}(u/L),$$ provided that $L$ is sufficiently large. In particular,
if $\gamma\subseteq B(R)$ is a smooth hypersurface, then for every $0\le i\le d-1$
\begin{equation}
\label{eq:bi exp scal preserve}
b_{i}(\gamma)=b_{i}(\exp_{x;L}(\gamma)),
\end{equation}
Further, for $r>0$ sufficiently small $\exp_{x}$ maps $B(r)$ into the geodesic ball $B_{x}(r)$, so that,
for every $R>0$, and $L$ sufficiently large, we have
\begin{equation}
\label{eq:exp map pres dist asymp}
\exp_{x;L}(B(R)) = B_{x}(R/L).
\end{equation}

We can then infer from \eqref{eq:bi exp scal preserve} combined with \eqref{eq:exp map pres dist asymp}, that
\begin{equation}
\label{eq:beta exp scal perturb}
\beta_{f_{x,L};i}(R)= \beta_{i}(f_{L};x,R/L)
\end{equation}
holds for every $R>0$, $L\gg 0$ sufficiently large. We observe that, by the assumption \eqref{eq:covar scal lim} of
Theorem \ref{thm:tot Betti numb loc}, the Gaussian random fields $\{f_{x,L}\}$ converge in law to the Gaussian random field $F_{x}$.
That alone does not ensure that one can compare the sample functions $f_{x,L}$ to the sample functions $F_{x}$,
without {\em coupling} them in a particular way, (i.e. define both on the same
probability space $\Omega$ to satisfy some postulated properties). Luckily, such a convenient coupling
was readily constructed ~\cite[Lemma 4]{sodin_lec_notes}, and we will reuse it for our purposes.

\vspace{2mm}

Our aim is to prove the following result, that, taking into account Theorem \ref{thm:Betti asymp Euclid}
applied on $F_{x}$, and \eqref{eq:beta exp scal perturb}, yields Theorem \ref{thm:tot Betti numb loc} at once.
We will denote $\Omega$ to be the underlying probability space, where all the random variables are going to be defined, and $\prob$
the associated probability measure.

\begin{proposition}
\label{prop:perturb Betti}
Under the assumptions of Theorem \ref{thm:tot Betti numb loc}, there exists a coupling of $F_{x}$ and $\{f_{x,L}\}$
so that for every $R>0$ and $\delta>0$ there exists a number $L_{0}=L_{0}(R,\delta)\in\Lc$ sufficiently big, so that
for all $L>L_{0}$ the following inequality holds outside an event of probability $<\delta$:
\begin{equation}
\label{eq:bi perturb}
\beta_{F_{x};i}(R-1) \le \beta_{f_{x,L};i}(R) \le \beta_{F_{x};i}(R+1).
\end{equation}
\end{proposition}

In what follows we are going to exhibit a construction of the small exceptional event from ~\cite{sodin_lec_notes},
where \eqref{eq:bi perturb} might not hold, prove by
way of construction that it is of arbitrarily small probability, and finally
culminate, this section with a proof that \eqref{eq:bi perturb} holds outside the exceptional event.

For $R>0$, $L\in \Lc$, $\alpha>0$ we denote the following ``bad" events in $\Omega$:
\begin{equation*}
\Delta_{1}= \Delta_{1}(R,L,\alpha) = \left\{\|f_{x,L}- F_{x}\|_{C^{1}(\overline{B}(2R))} > \alpha\right\},
\end{equation*}
%\begin{equation*}
%\Delta_{2}= \Delta_{2}(R,L,M) = \left\{\|f_{x,L}\|_{C^{2}(\overline{B}(2R))} > M\right\},
%\end{equation*}
%\begin{equation*}
%\Delta_{3}= \Delta_{3}(R,M) = \left\{\|F_{x}\|_{C^{2}(\overline{B}(2R))} > M\right\},
%\end{equation*}
and the ``unstable" event
\begin{equation*}
\Delta_{4} = \Delta_{4}(R,\alpha) = \left\{\min\limits_{y\in \overline{B}(2R)}\max\{ |F_{x}(y)|,|\nabla F_{x}(y) \} < 2\alpha \right\},
\end{equation*}
(with the more technical events $\Delta_{2},\Delta_{3}$ unnecessary for the purposes of this manuscript),
and then set the exceptional event
\begin{equation}
\label{eq:Delta except}
\Delta= \Delta(R,L,\alpha) :=\Delta_{1}\cup \Delta_{4}.
\end{equation}

The following bounds for the bad events are due to Nazarov-Sodin ~\cite{sodin_lec_notes}
(see also ~\cite{sarnak_wigman16,beliaev2018volume}).

\begin{lemma}
\label{lem:Deltai small}
There exists a coupling of $F_{x}$ and $\{f_{x,L}\}$ on $\Omega$, so that the following estimates hold.
\begin{description}

\item[a. ~{\cite[Lemma 4]{sodin_lec_notes}}] For every $R>0$, $\alpha>0$
\begin{equation*}
\limsup\limits_{L\rightarrow\infty} \prob\left(\Delta_{1}(R,L,\alpha)\right) = 0.
\end{equation*}

%\item[b. ~{\cite[\S 3.1]{sodin_lec_notes}}, ~{\cite[Lemma 4.4]{beliaev2018volume}}] For every $R>0$,
%\begin{equation*}
%\lim\limits_{M\rightarrow\infty}\limsup\limits_{L\rightarrow\infty}\prob\left(\Delta_{2}(R,L,M)\right) = 0.
%\end{equation*}

%\item[c. ~{\cite[\S 3.1]{sodin_lec_notes}}, ~{\cite[Lemma 4.4]{beliaev2018volume}}] For every $R>0$,
%\begin{equation*}
%\lim\limits_{M\rightarrow\infty}\prob\left(\Delta_{3}(R,M)\right) = 0.
%\end{equation*}

\item[b. ~{\cite[Lemma 5]{sodin_lec_notes}}] For every $R>0$,
\begin{equation*}
\lim\limits_{\alpha\rightarrow 0}\prob\left(\Delta_{4}(R,\alpha)\right) = 0.
\end{equation*}

\end{description}

\end{lemma}

The following lemma, due to Nazarov-Sodin, shows that if a function has no low lying critical points, then its nodal set is stable
under small perturbations.

\begin{lemma}[~{\cite[Lemmas 6-7]{sodin_lec_notes}}, ~{\cite[Proposition 6.8]{sarnak_wigman16}}]
\label{lem:func perturb comp}
Let $\alpha$, $R>1$, and $f:B(R)\rightarrow\R$ be a $C^{1}$-smooth function on an open ball $B=B(R)\subseteq\R^{d}$ for some $R>0$, such
that for every $y\in B(R)$, either $|f(y)|>\alpha$ or $\|\nabla f(y)|>\alpha$. Let $g\in C^{1}(B)$ such that
$\sup\limits_{y\in B}|f(y)-g(y)|<\alpha$. Then each nodal component $\gamma$ of $f^{-1}(0)$ lying in $B(R-1)$ generates
a nodal component $\gamma'$ of $g$ diffeomorphic to $\gamma$ lying in $B(R)$. Moreover, the map $\gamma\mapsto \gamma'$ between
the nodal components of $f$ lying in $B(R-1)$ and the nodal components of $g$ lying in $B(R)$ is injective.
\end{lemma}

We are now ready to show a proof of Proposition \ref{prop:perturb Betti}.

\begin{proof}[Proof of Proposition \ref{prop:perturb Betti}]

Let $R>0$ and $\delta>0$ be given. On an application of Lemma \ref{lem:Deltai small}b we obtain a number $\alpha=\alpha(R,\delta)$ so that
\begin{equation*}
\prob(\Delta_{4}(R,\alpha)) < \delta/2,
\end{equation*}
and subsequently, we apply Lemma \ref{lem:Deltai small}a to obtain number $L_{0}=L_{0}(R,\delta,\alpha)$ so that for all $L>L_{0}$,
\begin{equation*}
\prob(\Delta_{1}(R,L,\alpha)) < \delta/2.
\end{equation*}
%Continuing, we apply Lemma \ref{lem:Deltai small}b and Lemma \ref{lem:Deltai small}c, we obtain numbers $M=M(R,\delta)$ and
%$L_{1}=L_{1}(R,\delta)$ so that for all $L>L_{1}$
%\begin{equation*}
%\prob(\Delta_{2}(R,L,M)), \prob(\Delta_{3}(R,L,M)) < \delta/4,
%\end{equation*}
%and, by passing to the maximum if necessary, we may assume w.l.o.g. that $L_{1}>L_{0}$.
Defining the exceptional event as in \eqref{eq:Delta except}, the above shows that
\begin{equation*}
\prob(\Delta)<\delta.
\end{equation*}

We now claim that second inequality of \eqref{eq:bi perturb}
is satisfied on $\Omega\setminus\Delta$; by the above this is sufficient yielding the statement of
Proposition \ref{prop:perturb Betti}, and, as it was previously mentioned, also of Theorem \ref{thm:tot Betti numb loc}.
Outside of $\Delta$ we have both
\begin{equation*}
\min\limits_{y\in \overline{B}(2R)}\max\{ |F_{x}(y)|,|\nabla F_{x}(y) \} > 2\alpha
\end{equation*}
and
\begin{equation*}
\|f_{x,L}- F_{x}\|_{C^{1}(\overline{B}(2R))} < \alpha
\end{equation*}
for $L>L_{0}$, and these two also allow us to infer
\begin{equation*}
\min\limits_{y\in \overline{B}(2R)}\max\{ |f_{x,L}(y)|,|\nabla f_{x,L}(y) \} > \alpha
\end{equation*}
for $L>L_{0}$.
The first inequality of \eqref{eq:bi perturb} now follows upon a straightforward application of Lemma \ref{lem:func perturb comp},
with $F_{x}$ and $f_{x,L}$
taking the roles of $f$ and $g$ respectively, whereas the second inequality of \eqref{eq:bi perturb} follows upon reversing the roles of
$f$ and $g$. Proposition \ref{prop:perturb Betti} is now proved.

\end{proof}

\bibliographystyle{plainnat}
\bibliography{BettiNumbers_bib}

\end{document}